\documentclass[11pt]{article}
\usepackage[english]{babel}
\usepackage{amsmath, amssymb, latexsym, amsthm}
\usepackage{a4wide}
\usepackage{todonotes}
\usepackage{xcolor}
\usepackage{lipsum}

\newtheorem{theorem}{Theorem}[section]

\newtheorem{definition}[theorem]{Definition}

\newtheorem{problem}[theorem]{Problem}

\newtheorem*{warm-up}{Warm-up}

\usepackage{graphicx}
\usepackage{wrapfig}
\usepackage{tikz}
\usepackage[colorlinks=true,citecolor=black,linkcolor=black,urlcolor=blue]{hyperref}
\usetikzlibrary{arrows}
\usetikzlibrary{decorations.pathmorphing}
\usetikzlibrary{positioning}

\newcommand{\mH}{\mathcal{H}}
\newcommand{\mE}{\mathcal{E}}

\newcommand{\floor}[1]{\left \lfloor #1 \right \rfloor}

\title{A Note on the Minimum Number of Edges in Hypergraphs with Property O}

\author{Gal Kronenberg \thanks{School of Mathematical Sciences, Raymond and Beverly Sackler Faculty of Exact Sciences, Tel Aviv University, Tel Aviv, 6997801, Israel. E-mail: {\tt galkrone@mail.tau.ac.il}}  \and 
Christopher Kusch \thanks{Freie Universit\"at Berlin, Institut f\"ur Mathematik und Informatik, Arnimallee 3, 14195, Berlin, Germany and Berlin Mathematical School, Germany. E-mail: {\tt c.kusch@gmx.net}. Supported by a Berlin Mathematical School Phase II scholarship.} \and
Ander Lamaison \thanks{Freie Universit\"at Berlin and Berlin Mathematical School, Berlin, Germany. Email: {\tt lamaison@zedat.fu-berlin.de.} Supported by the Deutsche Forschungsgemeinschaft (DFG, German Research Foundation) under Germany's Excellence Strategy -- The Berlin Mathematics Research Center MATH+ (EXC-2046/1, project ID: 390685689).} \and 
Piotr Micek \thanks{Theoretical Computer Science Department, Faculty of Mathematics and Computer Science, Jagiellonian University, Krak\'ow, Poland, and Freie Universit\"at Berlin, Institut f\"ur Mathematik und Informatik, Arnimallee 3, 14195, Berlin, Germany. E-mail: {\tt piotr.micek@tcs.uj.edu.pl}. Partially supported by a Polish National Science Center grant (SONATA BIS 5; UMO-
2015/18/E/ST6/00299)} \and
Tuan Tran \thanks{ETH Zurich, Switzerland. E-mail: {\tt manh.tran@math.ethz.ch}. Research supported by the Czech Science Foundation, Grant number GJ16-07822Y, and with institutional support RVO:67985807.  This work has been done while TT was affiliated with the Institute of Computer Science of the Czech Academy of Sciences.}}
\date{}

\begin{document}

\maketitle

\begin{abstract}
	\noindent An oriented $k$-graph is said to have {\em Property O} if for every linear order of the vertex set, there is some edge oriented consistently with the linear order. Recently Duffus, Kay and R\"{o}dl investigated the minimum number $f(k)$ of edges in a $k$-uniform hypergaph
	with \emph{Property O}.
	They proved that $k! \leq f(k) \leq (k^2 \ln k) k!$, where the upper bound holds for sufficiently large $k$. In this short note we improve their upper bound by a factor of $k \ln k$ showing that $f(k) \le \left(\lfloor \frac{k}{2} \rfloor +1 \right) k! - \lfloor \frac{k}{2} \rfloor (k-1)!$ for every $k\geq 3$. We also show that their lower bound is not tight.
	Furthermore, Duffus, Kay and R\"{o}dl also studied the minimum possible number $n(k)$ of vertices in an oriented $k$-graph with Property O. For $k=3$ they showed that $n(3) \in \{6,7,8,9\}$, and asked for the precise value of $n(3)$. Here we show that $n(3)=6$.
\end{abstract}

%%%%%%%%%
%%%%%%%%%
%%%%%%%%% Introduction
%%%%%%%%%
%%%%%%%%%

\section{Introduction}
Extremal combinatorics is one of the central branches of modern combinatorial theory, which has developed spectacularly over the last few decades. A typical problem in extremal combinatorics has the following form: What is the maximum or minimum size of a finite structure which satisfies certain properties? One of the most famous examples is \emph{Property B}, first introduced by Bernstein in 1908 \cite{Bernstein} (see \cite{Probmeth} for a good overview), asking for the minimum number of edges in a $k$-uniform hypergraph such that every two colouring of its vertex set has a monochromatic edge.
 
Recently Duffus, Kay and R\"{o}dl \cite{DKR} introduced an analogue to Property B, called \emph{Property O}, with colouring replaced by order. 
An \emph{oriented $k$-graph} is a pair $\mH=(V, \mE)$, where $V$ is a finite set, and $\mE$ is a family of $k$-tuples of $V$ such that no $k$-tuple has a repeated element, and no two $k$-tuples induce the same set. We say $V$ is the vertex set of $\mH$, and $\mE$ is the edge set of $\mH$. In the case that $\mE$ contains a $k$-tuple for each $k$-subset of $V$, call $\mH$ a $k$-tournament. A tuple $(x_1,x_2,\ldots,x_{\ell})$ of distinct elements of $V$ is said to be consistent with a linear order  $<$ on $V$ if $x_1 < x_2 < \ldots < x_{\ell}$. 

\begin{definition}[Property O]
Given an oriented $k$-graph $\mH = (V, \mE)$, we say that $\mH$ has Property O if for every linear order $<$ of $V$, there exists an edge $\vec{e} \in \mE$ that is consistent with $<$. Furthermore, let $f(k)$ be the minimum number of edges in an oriented $k$-graph with Property O. That is,
\[ 
f(k) := \min \{|\mE | : \text{there exists an oriented $k$-graph $\mH =(V,\mE)$ with Property O}\}.
\]
\end{definition}

\noindent It is easy to check that $f(2)=3$ and an example for the upper bound is a cyclically ordered triangle. 
In \cite{DKR}, Duffus, Kay and R\"{o}dl initiated the study of $f(k)$, and established the following general bounds. 

\begin{theorem}[Duffus--Kay--R\"{o}dl]\label{thmDKR}
	The function $f(k)$ satisfies 
	\[
	k! \leq f(k) \leq (k^2 \ln k)k!,
	\] 
where the lower bound holds for all $k$ and the upper bound holds for $k$ sufficiently large.
\end{theorem}
Note that the lower bound follows from a simple counting argument, which we include here for convenience of the reader. Suppose that $\mH =([n],\mE)$ is an oriented $k$-graph that has Property O. Then every edge $\vec{e} \in \mE$ is consistent with exactly $\binom{n}{k} (n-k)! = \frac{n!}{k!}$ linear orders on $[n]$. Since $\mH$ has Property O, we must have $|\mE| \cdot \frac{n!}{k!} \geq n!$, implying $f(k) \geq k!$.

For the upper bound, Duffus et al. showed that a random $k$-tournament on $n=(k/e)^2(\pi \cdot \exp(e^2/2)\cdot k^3 \ln k)^{1/k}$ vertices with $\binom{n}{k}<k^2\ln(k)k!$ edges has Property O with positive probability. Furthermore they proved that almost all $k$-tournaments with $(1-o(1))\sqrt k\cdot  k!$ edges do not have Property O.

In this note we show that the lower bound in Theorem \ref{thmDKR} is not tight, and the upper bound in Theorem \ref{thmDKR} can be improved by a factor of $k \ln k$.

\begin{theorem}\label{UB}
For every integer $k\ge 3$, we have 
\[
k!+1 \le f(k) \leq \left(\floor{\frac{k}{2}}+1 \right) k! - \floor{\frac{k}{2}}(k-1)!.
\]
\end{theorem}
  
In Subsection~\ref{subsec:UB} we provide an explicit construction for showing the upper bound, and in Subsection~\ref{subsec:LB} we give a proof for the lower bound.

Another natural problem posed by Duffus, Kay and R\"{o}dl \cite{DKR} is to determine the minimum number of vertices in an oriented $k$-graph with Property O.
\begin{definition}
	For $k\geq 2$ we define 
\begin{equation*}
	n(k) := \min \{ |V| :\text{there exists an oriented } k \text{-graph }\mH =(V,\mE) \text{ having Property O}\}.
\end{equation*}
\end{definition}
Duffus et al. 
proved $6\leq n(3)\leq 9 $; for the upper bound they gave a construction and the lower bound was obtained via an exhaustive computer search. In Section~\ref{n(3)} we show $n(3)=6$ by providing two different constructions.

%%%%%%%%%%%%%%
%%%%%%%%%%%%%%
%%%%%%%%%%%%%% Proof of the upper bound
%%%%%%%%%%%%%%
%%%%%%%%%%%%%%

\section{Proof of Theorem~\ref{UB}}\label{sec:UB}

\subsection{Upper bound}\label{subsec:UB}
In this subsection we will construct an oriented $k$-graph with $\left(\floor{\frac{k}{2}}+1 \right) k! - \floor{\frac{k}{2}}(k-1)!$ edges possessing Property O. To aid the reader, we first describe the idea behind our construction for the case $k=3$. 

We start by defining two edges $(x,y,a)$ and $(y,x,b)$. Any ordering is consistent with the relative order of exactly one of these edges with respect to the positions of $x$ and $y$. If it happens to be $x<y$, but the edge $(x,y,a)$ is not consistent with the ordering, then there are two possibilities for the position of $a$ with respect to both $x$ and $y$. For each possibility we introduce one new vertex and two edges, such that at least one of them is consistent with the ordering. 

\begin{warm-up}\label{f(3)}
There is an oriented $3$-graph with $10$ edges having Property O.
\end{warm-up}

\begin{proof}
Let $\mH$ be an oriented $3$-graph with vertex set $V=\{ x,y, a,b,c,d,e,f \}$, and edge set
	 \begin{center}
	$\mE = \{(x,y,a),(a,x,c),(c,x,y),(x,a,d),(d,a,y),(y,x,b),(b,y,e),(e,y,x),(y,b,f),(f,b,x) \}.$
		\end{center}
Clearly $\mH$ has 10 edges. It remains to show that $\mH$ possesses Property O. Let $<$ be an arbitrary ordering of $V$. Since either $x< y$ or $y< x$, let us first suppose $x< y$. If the edge $(x,y,a)$ is not consistent with $<$, then we either have $a<x< y$ or $x<a<y$. If $a<x< y$, then
either $(a,x,c)$ or $(c,x,y)$ is consistent with $<$. On the other hand, if $x<a<y$, then one of the edges $(x,a,d)$ and $(d,a,y)$ is consistent with $<$.

Now, if $y<x$ but $(y,x,b)$ is not consistent with $<$, then either $b<y<x$ or $y<b<x$. In the first case one of the edges $(b,y,e)$  
and $(e,y,x)$ is consistent with $<$, and in the latter one of the edges $(y,b,f)$ and $(f,b,x)$ is consistent with $<$. Hence $\mH$ has Property O, 
as desired.
\end{proof}

To prove the upper bound in Theorem~\ref{UB}, we will generalise the above construction. We begin with $(k-1)+(k-1)!$ vertices $x_1,\ldots x_{k-1},a_1,\ldots,a_{(k-1)!}$, and $(k-1)!$ edges $\left(\pi_1,a_1\right),\ldots,\left(\pi_{(k-1)!},a_{(k-1)!}\right)$, where $\pi_1,\ldots,\pi_{(k-1)!}$ are permutations (viewed as $(k-1)$-tuples) of $\{x_1,\ldots,x_{k-1}\}$. Let $<$ be any linear order of the vertices. We can find $j \in [(k-1)!]$ such that $\pi_j$ is consistent with $<$. If the edge $(\pi_j,a_j)$ is not consistent with $<$, then there are $(k-1)$ possible locations for $a_j$ (in the restriction of $<$ to $x_1,\ldots,x_{k-1},a_i$). For each possibility we introduce one new vertex and  $\floor{\frac{k}{2}}+1$ edges such that at least one of the edges will be consistent with $<$. More details can be found below.  

\begin{proof}[Proof of Theorem~\ref{UB} (upper bound)]
We shall construct an oriented $k$-graph $\mH = (V,\mE)$ with the desired property. Let
\[
V=\left\{x_1,x_2,\ldots,x_{k-1},a_1,\ldots,a_{(k-1)!}, a_1^{(1)},a_1^{(2)},\ldots,a_1^{(k-1)},\ldots,a_{(k-1)!}^{(1)},a_{(k-1)!}^{(2)},\ldots,a_{(k-1)!}^{(k-1)}  \right\},
\] 
and let $\mE  $ be the set of $k$-tuples of $V$ defined as follows. Let $\pi_1,\ldots,\pi_{(k-1)!}$ be all possible permutations of $\{x_1,\ldots,x_{k-1}\}$ viewed as $(k-1)$-tuples.
\begin{itemize}
	\item[(1)] We put all $k$-tuples of the form $(\pi_j,a_j)$ into $\mE$. There are $(k-1)!$ such $k$-tuples.
	\item[(2)] For each $j \in [(k-1)!]$ we will add 
	$(k-1)\left( \floor{\frac{k}{2}}+1 \right)$ edges to $\mE$ as follows. For every $i \in [k-1]$ and $\ell \in \{1,3,\ldots, 2\floor{\frac{k}{2}}-1,k\}$, let $\pi_j^{(i)}$ be the $k$-tuple obtained by inserting $a_j$ into the $i$-th position of $\pi_j$, and let $\pi_j^{(i,\ell)}$ be the $k$-tuple obtained by replacing the $\ell$-th element of $\pi_j^{(i)}$ with $a_j^{(i)}$. For example, if $k=4, \pi_j=(x_1,x_3,x_2)$ and $i=2$, then $\pi_j^{(2)}=(x_1,a_j,x_3,x_2)$, $\pi_j^{(2,1)}=(a_j^{(2)},a_j,x_3,x_2)$, $\pi_j^{(2,3)}=(x_1,a_j,a_j^{(2)},x_2)$ and $\pi_j^{(2,4)}=(x_1,a_j,x_3,a_j^{(2)})$. Finally, we put $\pi_j^{(i,\ell)}$ into $\mE$.

\end{itemize}

It is clear that $\mH=(V,\mE)$ is an oriented $k$-graph with $|V|=k!+k-1$ and	
\[
|\mE | = (k-1)!+ (k-1) \left( \floor{\frac{k}{2}} +1 \right)\cdot (k-1)! =  \left( \floor{\frac{k}{2}} +1 \right) k! - \floor{\frac{k}{2}}(k-1)!.
\]
To see that $\mH$ has Property O, let $<$ be an arbitrary linear order on $V$. Since $\{\pi_1,\ldots,\pi_{(k-1)!}\}$ is the set of all permutations of $\{x_1,\ldots,x_{k-1}\}$, there exists $j \in [(k-1)!]$ such that $\pi_j$ is consistent with $<$. If the edge $(\pi_j,a_j)$ is consistent with $<$, then we are done. Otherwise, there is $i \in [k-1]$ such that $\pi_j^{(i)}$ is consistent with $<$. By a straightforward but slightly tedious case-by-case analysis, one can show that one of the edges $\pi_{j}^{(i,\ell)}$, $\ell \in \{1,3,\ldots, 2\floor{\frac{k}{2}}-1,k\}$, will be consistent with $<$.
\end{proof}

\subsection{Lower bound}\label{subsec:LB}
In this subsection we will prove the lower bound $f(k) \ge k!+1$.

\begin{proof}[Proof of Theorem \ref{UB} (lower bound)] 
	Suppose for the contrary that there exists an oriented $k$-graph $\mH=([n],\mE)$ with Property O such that $|\mE| \le k!$. Observe that each edge $\vec{e} \in \mE$ is consistent with exactly $\binom{n}{k}(n-k)! = \frac{n!}{k!}$ linear orders of $[n]$. Since $\mH$ has Property O, we must have $|\mE| \cdot \frac{n!}{k!} \geq n!$. Hence $|\mE|= k!$, and each linear order of $[n]$ is consistent with exactly one edge.
	
	Write $\mE=\{\vec{e_1},\ldots,\vec{e_{k!}}\}$. Without loss of generality we can suppose that $\vec{e_1}=(1,2,\ldots,k)$. We associate to each permutation $\sigma:[k]\rightarrow [k]$ a linear order $\bar{\sigma}$ on $[n]$, defined by
	\[
	\sigma(1)\overset{\bar{\sigma}}{<}\sigma(2) \overset{\bar{\sigma}}{<}\ldots \overset{\bar{\sigma}}{<}\sigma(k) \overset{\bar{\sigma}}{<}k+1 \overset{\bar{\sigma}}{<} k+2 \overset{\bar{\sigma}}{<}\ldots \overset{\bar{\sigma}}{<} n.
	\]
	For $1 \le i \le k!$, let $S_i=\{\sigma:\vec{e_i} \enskip \text{is consistent with} \enskip \bar{\sigma}\}$. From the discussion above, we learn that each permutation $\sigma: [k] \rightarrow [k]$ is contained in exactly one $S_i$, and so
	\[
	\sum_{1 \le i \le k!}|S_i|=k!.
	\]
	Note that
	\[
	|S_i|=\frac{k!}{(|\vec{e_i}\cap [k]|)!} \enskip \text{for every $i \in [k!]$ such that $S_i \ne \emptyset$,}
	\]
	since whether $\sigma \in S_i$ depends only on the relative order of $\vec{e_i}\cap [k]$. This together with the assumption that $\vec{e_1}$ and $\vec{e_i}$ induce different $k$-sets implies that $|S_1|=1$ and $|S_i|$ is divisible by $k$ for $i>1$. Therefore, we get a contradiction
	\[
	0\equiv k!=\sum_{1 \le i \le k!} |S_i|\equiv 1 \pmod k,
	\]
	finishing our proof.
\end{proof}

%%%%%%%%%%%%%%
%%%%%%%%%%%%%%
%%%%%%%%%%%%%% n(3) section
%%%%%%%%%%%%%%
%%%%%%%%%%%%%%

\section{Oriented \texorpdfstring{$3$}{3}-graphs with Property O}\label{n(3)}
In this section we will construct two oriented $3$-graphs on $6$ vertices having Property O. Combined with the lower bound $n(3) \ge 6$ given in \cite{DKR}, this shows $n(3)=6$.

\paragraph{First construction}
For each $k\ge 2$ we construct an oriented $k$-graph $\mH_k=(V_k,\mE_k)$ with Property O, where 
\begin{equation}\label{eq:size}
|V_k|=3\cdot 2^{k-2} \enskip \text{and} \enskip |\mE_k|=3^{k-1}\cdot 2^{\binom{k-1}{2}}.     
\end{equation}
Let $\mH_2=(V_2,\mE_2)$ be an oriented $3$-cycle. It is easy to see that $\mH_2$ has Property O, and the sizes of its vertex set and edge set are given by \eqref{eq:size}. 

Suppose that we have constructed an oriented $k$-graph $\mH_k=(V_k,\mE_k)$ with the desired properties. Now make two disjoint copies $\mH_X = (X, \mE_X)$ and $\mH_Y=(Y,\mE_Y)$ of $\mH_k=(V_k,\mE_k)$.
Let $\mH_{k+1}=(V_{k+1},\mE_{k+1})$ be an oriented $(k+1)$-graph with vertex set $V_{k+1}=X\cup Y$, and edge set
\[
\mE_{k+1}=\left\{(x_1,\ldots,x_k,y):y \in Y, (x_1,\ldots,x_k)\in \mE_X \right\} \cup \left\{(y_1,\ldots,y_k,x):x \in X, (y_1,\ldots,y_k)\in \mE_Y\right\}
\]
(see Fig. 1).

\begin{center}

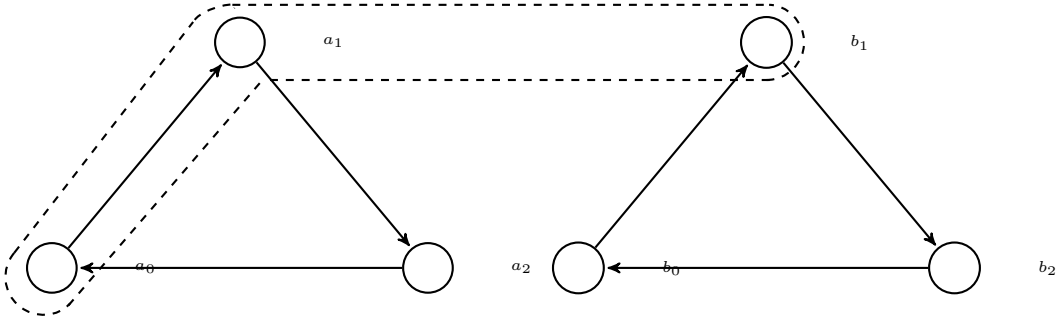
\begin{figure}[ht] \label{fig_n(3)}

\begin{tikzpicture}[-,>=stealth',shorten >=1pt,auto,node distance=5cm,
      thick,main node/.style={circle,fill=none,draw,font=\sffamily\tiny\bfseries}]
\hspace*{0.08\linewidth}

      \node[main node] (1) at (0,0) {$a_0$};
      \node[main node] (2) at (5,0) {$a_2$};
      \node[main node] (3) at (7,0) {$b_0$};
      \node[main node] (4) at (12,0) {$b_2$};
      \node[main node] (5) at (2.5,3) {$a_1$};
      \node[main node] (6) at (9.5,3) {$b_1$};
      
     \draw[->] (1) -- (5) node[pos=.5,above] {};
     \draw[->] (5) -- (2) node[pos=.5,right] {};
     \draw[->] (2) -- (1) node[pos=.5,above] {};
     \draw[->] (4) -- (3) node[pos=.5,right] {};
     \draw[->] (3) -- (6) node[pos=.5,above] {};
     \draw[->] (6) -- (4) node[pos=.5,right] {};
     
     \draw[dashed] (-0.5,0.2) arc (140:320:0.5);
     \draw[dashed] (9.5,2.5) arc (-90:90:0.5);

     \draw[dashed] (-0.5,0.2) --(1.9, 3.3) node[pos=.5,left] {};
     \draw[dashed] (0.22,-0.5) --(2.8,2.5) node[pos=.5,left] {};
     \draw[dashed] (2.4, 3.5) --(9.6,3.5) node[pos=.5,left] {};
     \draw[dashed] (2.9,2.5) --(9.6,2.5) node[pos=.5,left] {};
     
     \draw[dashed] (1.9,3.3) to [out=55,in=45] (2.4,3.43);

    \end{tikzpicture}
    \caption{The oriented $3$-graph $\mathcal{H}_3$ with the edge $(a_0,a_1,b_1)$ depicted.}
\end{figure}
\end{center}
\vspace*{-1cm}

Let us show that $\mH_{k+1}=(V_{k+1},\mE_{k+1})$ satisfies \eqref{eq:size}. Indeed, from the definition of $V_{k+1}$ and $\mE_{k+1}$ we find $|V_{k+1}|=2|V_k|=2\cdot 3\cdot 2^{k-2}=3\cdot 2^{k-1}$, and
\[
|\mE_{k+1}|=|V_{k+1}|\cdot |\mE_k|=3\cdot 2^{k-1}\cdot 3^{k-1}\cdot 2^{\binom{k-1}{2}}=3^{k}\cdot 2^{\binom{k}{2}}.
\]

To see that $\mH_{k+1}$ has Property O, let $<$ be any linear order on $V_{k+1}$. We can find an edge of $\mE_{k+1}$ which is consisten with $<$ as follows.
\begin{itemize}
    \item[(1)] Suppose first that there is some vertex $y \in Y$ such that $y>x$ for all $x \in X$. Because $\mH_X$ is isomorphic to $\mH_k$, some tuple $(x_1,\ldots,x_k) \in \mE_X$ must be consistent with $<$. Since $y>x_k$, the edge $(x_1,\ldots,x_k,y)$ is consistent with $<$.
    \item[(2)] Suppose now that there exists $x \in X$ such that $x>y$ for every $y \in Y$. By the same argument as above, we can show that some edge of the form $(y_1,\ldots,y_k,x)$ is consistent with $<$.
\end{itemize}

\paragraph{Second construction} The second example is a simple modification of the construction given in Section \ref{sec:UB}. Instead of using the vertices $e$ and $f$, we can use $c$ and $d$ again: Simply replace $e$ by $d$ and $f$ by $c$. So we get an oriented $3$-graph $\mH$ with vertex set $V= \{ x,y,a,b,c,d \}$ and edge set
	 \begin{center}
	$\mE = \{(x,y,a),(a,x,c),(c,x,y),(x,a,d),(d,a,y),(y,x,b),(b,y,d),(d,y,x),(y,b,c),(c,b,x) \}.$
		\end{center}
	One can use the same proof to show that 
	$\mH$ possesses Property O. 
	
%%%%%%%%%%%%%%
%%%%%%%%%%%%%%
%%%%%%%%%%%%%% Concluding Remarks
%%%%%%%%%%%%%%
%%%%%%%%%%%%%%

\section{Concluding Remarks}
In this note we have shown that $k!+1\le f(k)\leq \left(\floor{\frac{k}{2}}+1 \right) k! - \floor{\frac{k}{2}}(k-1)!$ for every $k\geq 3$. 
The main open problem regarding the asymptotic behaviour of $f(k)$ is the following.
\begin{problem}[Duffus--Kay--R\"{o}dl \cite{DKR}]
	Is it true that $\frac{f(k)}{k!} \rightarrow \infty $ as $k\rightarrow\infty$?
\end{problem}

\noindent We  believe that the answer should be yes. However, we have not even been able to show that there is some absolute constant $c>1$ such that $f(k)>ck!$ for $k$ sufficiently large. Of course, any improvement of the upper bound would be interesting as well.

Recall that $n(k)$ is the minimum possible number of vertices in an oriented $k$-graph with Property O. It is easy to see that $n(2)=3$, while showing $n(3)=6$ requires some effort (see Section \ref{n(3)}). This is all we know about the precise values of $\{n(k)\}_{k\ge 2}$. However, for large $k$, one can determine $n(k)$ asymptotically.
Indeed, a trivial lower bound on $n(k)$ is $\left(\frac{k}{e}\right)^2$, since the number of edges is at least $k!$ and at most $\binom{n(k)}{k}$. On the other hand, Duffus et al. \cite[pp.~3--4]{DKR} showed that a random $k$-tournament on $n=\left(\frac{k}{e}\right)^2(\pi\cdot \exp(e^2/2)\cdot k^3\ln k)^{1/k}$ vertices has Property O with positive probability. Hence 
\[
n(k)\le \left(\frac{k}{e}\right)^2(\pi\cdot \exp(e^2/2)\cdot k^3\ln k)^{1/k}=(1+o(1))\left(\frac{k}{e}\right)^2,
\]
and so $n(k)=(1+o(1))\left(\frac{k}{e}\right)^2$.

The second construction in Section \ref{n(3)} has fairly few edges, namely 10, and $n(3)=6$ vertices. This naturally leads us to the following question.

\begin{problem}
Let $k\geq 3$ be an integer. Is there an oriented $k$-graph with $n(k)$ vertices and $f(k)$ edges having Property O?
\end{problem}

\subsection*{Acknowledgment}

The authors would like to thank the participants (in particular to Shagnik Das) of the ``Open Problems in Combinatorics and Graph Theory 2016" workshop of the FU Berlin Combinatorics and Graph Theory group for stimulating discussions on the topic, as well as Tibor Szab\'o and the second author for organising it.

\end{document}